\documentclass[a4paper,12pt,reqno]{amsart}

\numberwithin{equation}{section}

\usepackage{hyperref}
\hypersetup{
	colorlinks=true,
	linkcolor=blue,
	citecolor=red,
	pdfstartview=,
}
\usepackage[dvipsnames]{xcolor}
\usepackage{cite}
\usepackage{amsfonts,amsmath,amssymb,amsfonts,amsthm}
\usepackage[left=2.6cm,right=2.6cm,top=2.5cm,bottom=2.5cm,twoside]{geometry}

\usepackage{times}

{
\newtheorem{thm}{Theorem}[section]
\newtheorem{prop}[thm]{Proposition}
\newtheorem{cor}[thm]{Corollary}
\newtheorem{lemma}{Lemma}[section]
\newtheorem{problem}{Problem}
\newtheorem{conjecture}{Conjecture}
}

	\newtheorem{theorem}{Theorem}

{\theoremstyle{definition}

\newtheorem{example}{Example}[section]
\newtheorem{rem}{Remark}[section]
}

\title[On the periodicity of entire functions]{ New findings on the periodicity of entire functions and their differential polynomials}
\author{M.~A.~Zemirni, I.~Laine and Z.~Latreuch }

\address[M.~A.~Zemirni]{}
\email{\href{mailto:zmrn.amine@gmail.com}{zmrn.amine@gmail.com}}

\address[I.~Laine]{Department of Physics and Mathematics, University of Eastern Finland,
	P.O.Box~111, 80101 Joensuu, Finland.}
\email{\href{ilpo.laine@uef.fi}{ilpo.laine@uef.fi}}

\address[Z.~Latreuch]{National Higher School of Mathematics, Sidi Abdellah, Algeria.}
\email{\href{z.latreuch@nhsm.edu.dz}{z.latreuch@nhsm.edu.dz}}



\begin{document}
	
\begin{abstract}In this paper, we seek to explore under what conditions the periodicity of an entire function $ f(z) $ follows from the periodicity of a differential polynomial in $ f(z) $. We improve and generalize some earlier results and we give  other new findings. 
\end{abstract}
	
\maketitle

%
%

\section{Introduction}
Recently, there has been renewed interest in the periodicity of entire functions and differential polynomials. This can be attributed to the development on the application of Nevanlinna theory into the theory of complex differential and difference equations. Here we assume that the reader is familiar with the fundamental results and standard notations of Nevanlinna theory \cite{L,LLY1,YY}, including the notions of order, hyper-order and convergence exponent of zeros. For instance,  a meromorphic function $\varphi(z)$ is said to be a small function of a meromorphic function $f(z)$ if the Nevanlinna characteristic $T(r,\varphi)$ of $ \varphi(z) $ satisfies $T(r,\varphi)=S(r,f)$, where the notation $S(r,f)$ stands for any quantity satisfying $S(r,f)=o(T(r,f))$ as $r\to \infty$ possibly outside an exceptional set of finite linear, resp. logarithmic, measure. 
In addition, a periodic function $f(z)$ with fundamental period $c$ or $mc$, for some integer $m\in \mathbb{N}$,  is said to be $c$-periodic.

Our starting point is the following conjecture.
\begin{conjecture}[Generalized Yang's conjecture]\label{Con}
	Let $f(z)$ be a transcendental entire function and $n,k$ be positive integers. If $f(z)^n f^{(k)}(z)$ is a periodic function, then $f(z)$ is also a periodic function.
\end{conjecture}
This conjecture originally appeared in \cite{LLY,WH} for $ n = 1 $, known
as Yang's Conjecture. The present form of Conjecture~\ref{Con} has been stated in \cite{LWY}. It has been verified in the case $k=1$ in \cite{WH,LWY}. Furthermore, Conjecture~\ref{Con} remains true for entire functions having a Picard exceptional value with some restrictions on the growth \cite{LWY,LZ}. 

The origin of this conjecture goes back to $ 1939 $, when Titchmarsh \cite[p.~267]{T} showed that
the differential equation 
\begin{equation*}
f(z) f''(z) = -\sin^2(z)
\end{equation*}
has no real entire solutions of finite order other than $ f(z)=\pm \sin(z) $. After that, Li, L\"u and Yang  \cite[Theorem~1]{LLY} removed the assumption of $ f(z) $ being real entire function in Titchmarsh's result.  This, in fact, leads to the above conjecture after seeing that the differential polynomial $ -\sin^2(z) $ and the solutions $ f(z)=\pm\sin(z) $ are both periodic functions.  This property of the periodicity  appears in certain general non-linear differential equations. For example, the periodic function $ f(z)=e^{z / 4}+e^{-z / 4} $ satisfies the differential equation
\begin{equation*}
f(z)^{4}-64 f(z) f^{\prime \prime}(z)+2=e^{z}+e^{-z}.
\end{equation*}
Also, we see that the periodicity phenomenon occurs for meromorphic functions as well. For example, the periodic function $ f(z)=e^{2 z} /\left(e^{z}-1\right) $ satisfies the differential equation
\begin{equation*}
f(z)^{2}+f^{\prime}(z)-3 f(z)=e^{2 z}.
\end{equation*}
These observations lead to the question of whether the above conjecture is still valid if $f(z)^nf^{(k)}(z)$ is replaced by a general differential polynomial $P(z,f)$ with small coefficients. As shown by the following example, the answer to this question is not always affirmative.
\begin{example}\label{exam1}
	(1)  The function $f(z)=\exp(e^{2\pi iz}-z)$ is not periodic whereas the polynomial $ P(z,f) := e^{ 2z} f(z)^2 + e^{z} f(z)$ is periodic.\\[1ex]
	(2)   The function $f(z)=ze^z+d$, where $ d $ is a constant,  is not periodic whereas the differential polynomial $
	P(z,f):=(f'(z))^2-f(z)f''(z)+df(z)=\left(e^{z}-d\right)^2
	$ is periodic.
\end{example}

Therefore, the natural way to deal with the aforementioned question is to consider the following problem, instead.
\begin{problem}\label{pb}
	{Under what conditions the periodicity of a differential polynomial $ P(z,f) $ implies that of $ f(z) $?}
\end{problem}

This problem has been considered, for example in \cite{LWY,LY,LZ,LZ2,WLL}, for $ P(z,f) $ having the following specific forms:
\begin{equation*}
f^{(k)}(z) Q(f(z)) , \quad Q(f(z))^{(k)}, \quad f(z)^n + L(f(z)),
\end{equation*} 
where $ Q(z) $ is a polynomial, $ L(f) $ is a linear differential polynomial. 
In contrast to the above, our approach in treating Problem~\ref{pb} allows us to consider a more general form of $ P(z,f) $. Consequently, we give generalizations to the earlier results as well as new findings.


This paper is organized as follows. In Section~\ref{Sec2}, we give our main results discussing the Problem~\ref{pb} in cases when $ P(z,f) $ is a differential monomial and a differential polynomial with at least two non-zero terms and with constant coefficients. Some further results are given in Section~\ref{Sec3}, where we discuss the periodicity of some meromorphic functions and the periodicity of polynomials in an entire function with non-constant coefficients. The results of Section~\ref{Sec3} will be used to prove the results of Section~\ref{Sec2}. Results in Section~\ref{Sec3} can also be seen of independent interest. Finally, the proofs of the results of Section~\ref{Sec2} are given in Section~\ref{Sec 4}.

%
%

\section{Periodicity of differential polynomials}\label{Sec2}
This section is devoted to stating our main results regarding Problem~\ref{pb}.  In fact, we mainly focus on the case when $ f(z) $ is a transcendental entire function.  The case of meromorphic functions is still open for further research.

Here, we give some conditions on $ f(z) $ so that the implication in Problem~\ref{pb} holds. Besides, when $ P(z,f) $ reduces to a differential monomial, we find that no condition concerning the form of $ P(z,f) $ is needed. However, for differential polynomials with at least two terms, some conditions on $ P(z,f) $ are required. 


\subsection{Differential monomials}
We recall the following result.

\begin{theorem}[{\cite[Theorem~3.1]{LZ}}]\label{LZ1}
	Let $f(z)$ be a transcendental entire function and $n, k$ be positive integers. Suppose that  $f(z)^{n} f^{(k)}(z)$ is a periodic function with period $ c $, and one of the following holds:
	\begin{itemize}
		\item[(i)] $ f(z) $ has the value $ 0 $ as a Picard exceptional value, and $ \rho_2(f)<\infty $.
		\item[(ii)] $ f(z) $ has a nonzero Picard exceptional value.
	\end{itemize}
	Then $ f(z) $ is periodic of period $ c $ or $ (n+1) c $.
\end{theorem}
As for the case when $f(z)$ has, in particular, a finite Borel exceptional value, we mention the following result.
\begin{theorem}[{\cite[Theorem~1.1]{LZ2}}]\label{LZ2}
	Let $f(z)$ be a transcendental entire function with $\rho_2(f)<~1$ and $n, k$ be positive integers.
	If there is a constant $ d$ such that $ \lambda(f-d)< \rho(f) \le~\infty $ and $f(z)^{n} f^{(k)}(z)$ is a periodic function, then $f(z)$ is a periodic function as well.
\end{theorem}

The main objective of this subsection is to improve and generalize Theorems~\ref{LZ1}~and~\ref{LZ2} by replacing $ f(z)^{n} f^{(k)}(z) $ with an arbitrary differential monomial.

A differential monomial $ M(z,f) $ generated by $ f(z) $ is a finite product of $f(z)$ and its derivatives, that is,
\begin{equation}\label{mo}
M(z,f)=f(z)^{\lambda_{0}}\left(f^{\prime}(z)\right)^{\lambda_{1}} \cdots\left(f^{(n)}(z)\right)^{\lambda_{n}},
\end{equation}
where $\lambda_{0},\ldots,\lambda_{n}$ are non-negative integers. 
The quantities 
\begin{equation} \label{dw}
\gamma_M:=\lambda_{0}+\cdots+\lambda_{n}\quad \text{and} \quad  \Gamma_M:=\lambda_{1}+2\lambda_{2}+\cdots+n\lambda_{n}
\end{equation} 
are called the degree and the weight of $ M(z,f) $, respectively.

Our first result deals with the case when $ f(z) $ has, in particular, the value~$ 0 $ as a Borel exceptional value.

\begin{thm}\label{th3}
	Let $ f(z) $ be a transcendental entire function with $ \lambda(f)< \rho(f) \le \infty $, and suppose that $ M(z,f) $ is a periodic function with period $ c $. Then the following holds:
	\begin{enumerate}
		\item If $ \rho_2(f)<1 $, then $ f(z)=e^{az+b} $, where $ a,b $ are non-zero constants and $ e^{\gamma_M ac}=1 $.
		\item If $ 1\le\rho_2(f)< \infty$ and $ \lambda(f) < \rho_2(f)$, then $ f(z) $ is $ c $-periodic.
	\end{enumerate}
\end{thm}

Clearly, if an entire function $ f(z) $ is periodic, then $ f^{(k)}(z) $ must be periodic. The converse is not always true, as shown by the function $ f(z)=e^z+z $. However, the following consequence of Theorem~\ref{th3} shows, under certain condition, that the converse holds.

\begin{cor}
Let $ f(z) $ be a transcendental entire function, and suppose there exists a constant $ d\neq0 $ such that $ \lambda(f-d)< \rho(f) \le \infty $. If $f^{(k)}(z) $ is a periodic function with period $ c $, where $ k\in\mathbb{N} $, then the conclusions of Theorem~\ref{th3} hold.
\end{cor}
	
This can be proved by taking $ g(z)=f(z)-d $ and $ M(z,g) =g^{(k)}(z)$. The rest follows immediately.

In contrast to Theorem~\ref{th3},  the next result shows that when $ f(z) $ has, in particular, a nonzero Borel exceptional value, the condition ``$ \lambda(f) < \rho_2(f)$" is not needed, and there is no restriction on the  growth of $ f(z) $.

\begin{thm}\label{th4}
	Let $ f(z) $ be a transcendental entire function, and suppose there exists a constant $ d\neq0 $ such that $ \lambda(f-d)< \rho(f) \le \infty $. If $ M(z,f) $ is a periodic function with period~$ c $ and $ \lambda_{0}>0 $, then $ f(z) $ is $ c $-periodic.
\end{thm}

The conclusion in Theorem~\ref{th3}(1) means clearly that $f(z)$ is periodic of period $c$ or~$\gamma_Mc$. Furthermore, from the proofs of Theorems~\ref{th3} and \ref{th4}, the conclusion of $ f(z) $ being $ c $-periodic also means that $f(z)$ is periodic of period $c$ or $\gamma_Mc$.

\subsection{Differential polynomials with at least two non-zero terms}

Define the differential polynomial  $ P(z,f) $ by
\begin{eqnarray}\label{dif}
P(z,f)=\sum_{j=1}^{m}\alpha_jM_j(z,f).
\end{eqnarray}
Here $\alpha_j$ are non-zero constants and $ M_j(z,f) $ are differential monomials given by 
\begin{equation}\label{mo2}
M_j(z,f)=f(z)^{\lambda_{0j}}\left(f^{\prime}(z)\right)^{\lambda_{1j}} \cdots\left(f^{(n)}(z)\right)^{\lambda_{nj}},
\end{equation}
where $\lambda_{0j},\ldots,\lambda_{nj}$ are  non-negative integers. 
 
Let $\gamma_1,\gamma_2,\ldots,\gamma_m$ be the degrees of  $M_1(z,f), M_2(z,f), \ldots, M_m(z,f)$, respectively. Similarly, let  $\Gamma_1,\Gamma_2,\ldots,\Gamma_m$ be the weights of $M_1(z,f), M_2(z,f), \ldots, M_m(z,f)$, respectively. Both degree and weight of differential monomials  are given in \eqref{dw}. The total degree $\gamma_{P}$ and total weight $\Gamma_P$ of $P(z,f)$ are defined, respectively, by
$$
\gamma_{P}=\max_{1\leq j\leq m}\gamma_{j} \quad \text{and}\quad  \Gamma_P=\max_{1\leq j\leq m}\Gamma_{j}.
$$
Since the differential polynomial $P(z,f)$ is supposed to be periodic, we may assume, without loss of generality, that $\gamma_{j}>0$, for all $j\in \{1,\ldots,m\}$.  

Before stating our main results, we need to fix some notations. 
Define the sequence of positive integers $ \delta_1, \delta_2, \ldots, \delta_l $ as follows: 
	\begin{equation}\label{seq}\begin{split}
	\delta_1&=\min_j\gamma_j , \\
	\delta_2 &= \min_j\{\gamma_j: \gamma_j\neq \delta_1 \}, \\
	&\; \vdots\\
	\delta_l &= \min_j\{\gamma_j: \gamma_j\neq \delta_i, i=1, \ldots, l-1 \}= \gamma_{P} . 
	\end{split}
	\end{equation}
Clearly $ 1\le l\le m $ and the sequence $ \{\delta_i\} $ is strictly increasing.
We denote by $\Lambda(\delta_i)$ the set that contains the indices of the terms in \eqref{dif} with the highest weights among those of degree~$\delta_i$. Now, we define the subset $\Lambda_P\subset \{\delta_1,\delta_2,\ldots,\delta_l\}$~by
$$
\Lambda_P=\left\lbrace \delta_i:\sum_{k\in\Lambda(\delta_i)}\alpha_k\neq 0\right\rbrace .
$$
The set $ \Lambda_P $ is illustrated by the following differential polynomial
	$$
	P(z,f)=f''(z)f(z)^2-2(f'(z))^2f(z)-f'(z)f(z)^2+f''(z)f(z)-(f'(z))^2+f''(z)+f'(z),
	$$
where $\delta_1=1$, $ \delta_2=2 $ and $\delta_3=\gamma_P=3$. Here, we have  $\Lambda_{P}=\{1,3\}$, since
$$
\sum_{k\in\Lambda(1)}\alpha_k=1,\quad \sum_{k\in\Lambda(2)}\alpha_k=1-1=0 \quad\text{and}\quad \sum_{k\in\Lambda(3)}\alpha_k=1-2=-1.
$$

In this subsection, we discuss the periodicity of $f(z)$ according to its hyper-order. To start with, we assume that the hyper-order of $f(z)$ is $ <1 $.

\begin{thm}\label{th6}
	Let $f(z)$ be a transcendental entire function with $\rho_2(f)<1$, and suppose that there exists $ d\in\mathbb{C} $ such that $ \lambda(f-d)<\rho(f)\le \infty $. Suppose that $P(z,f)\not\equiv 0$ is periodic with period $c$, $\lambda_{0j}>0$ for every $j\in \{1,\ldots,m\}$ and one of the following holds
	\begin{itemize}
		\item [(i)] $d=0$,
		\item [(ii)] $d\neq 0$ and $\lambda_{01}=\cdots=\lambda_{0m}=\lambda>0$,
		\item[(iii)] $d\neq 0$ and $ \gamma_P\in \Lambda_{P} $.
	\end{itemize}
	Then $ f(z) $ is $ c $-periodic.
\end{thm}

The conclusion of Theorem~\ref{th6} may not hold if $P(z,f)\equiv 0$. For example, the function $f(z)=e^{z^2}$ is not periodic and 
$$
P(z,f)=f'(z)^2f(z)-f''(z)f(z)^2+2f(z)^3\equiv 0.
$$

Example~\ref{exam1}(2) shows that the conclusion of Theorem~\ref{th6} may not hold if the function $ f(z) $ does not appear in each term of $ P(z,f) $, i.e., $ \lambda_{0j}=0 $ for some $ j $.

Theorem~\ref{th6} improves \cite[Theorem~1]{CLW} in the sense that in our result, the differential polynomial does not need to have one term only with the highest degree.

From the proof of Theorem~\ref{th6}, we see that no condition on $\lambda_{0j}$ and $\gamma_{P}$ is needed when~$d$ is a Picard exceptional value of $ f(z) $. We state this without proof as follows.

\begin{cor}\label{Th4.2}
	Let $f(z)$ be a transcendental entire function with a finite Picard exceptional value $d$ and $\rho_2(f)<1$. Suppose that $P(z,f)\not\equiv0$ is a periodic function with period~$c$. Then $ f(z) $ is $ c $-periodic.
\end{cor}

From Corollary~\ref{Cor} below and from the proof of Theorem~\ref{th6} we see that $ f(z) $ in the conclusion of Theorem~\ref{th6} and Corollary~\ref{Th4.2} has the form $f(z)=e^{az+b} +d$, where $ a,b $ are non-zero constants and $ a $ satisfies $ e^{\delta ac} =1 $, where $\delta$ is an integer not exceeding~$\gamma_{P}$.

\medskip
Next, we present some sufficient conditions that ensure the periodicity of $f(z)$ when its hyper-order is $ \ge1 $.

\begin{thm}\label{th7}
	Let $f(z)$ be a transcendental entire function with $1\le\rho_2(f)<\infty$, and suppose that there exists $ d\in\mathbb{C} $ such that $ \lambda(f-d)<\rho_2(f) $. Suppose that $P(z,f)$ is a periodic function with period $c$, $\Lambda_P\neq \emptyset$ and one of the following holds:
	\begin{itemize}
		\item [(i)] $d=0$;
		\item[(ii)] $d\neq 0$ and $\gamma_{P}\in \Lambda_P$.
	\end{itemize}
	Then $f(z)$ is $ c $-periodic.
\end{thm}

From the proof of Theorem~\ref{th7}, one may see that the condition ``$ \gamma_{P}\in \Lambda_P $" is not needed  if $\lambda_{0j}=0$ for all $j\in \{1,\ldots,m\}$.

The condition $\Lambda_P \neq \emptyset$ is a sufficient, but not a necessary condition for the periodicity~of~$f(z)$. The following example illustrates this point. For example, the function $f(z)=e^{\sin z}$ satisfies  $ \lambda(f)=0<1=\rho_2(f) $ and solves the differential equation
$$
P(z,f)=(f'(z))^2-f''(z)f(z)=e^{2\sin z}\sin z.
$$
Here, $P(z,f)$ and $f(z)$ are both periodic while $\Lambda_P = \emptyset$.

%
%
\section{Further results }\label{Sec3}
In this section, we give some observations on the periodicity of meromorphic functions as well as on the periodicity of polynomials in an entire function $ f(z) $ with non-constant coefficients.  The results presented here will be used later to prove the main theorems of Section~\ref{Sec2}, and are also of independent interest.

\subsection{Periodicity of some meromorphic functions}

We start with following basic result.
\begin{prop}\label{period-function-1}
	Let $ v(z)\not\equiv 0 $ be a meromorphic function of order $ \rho(v)<\infty $, and $ g(z)$ be a non-constant entire function. If $ F(z) = v(z)e^{g(z)} $ is a periodic function of period~$ \tau $, then either $ \rho(g)\ge 1 $, or $ g(z) $ is polynomial. In the case when $ g(z) $ is polynomial, we have
	$$
	\left\{
	\begin{array}{lcl}
	\rho(v) \ge\deg(g),& \text{ if}& \deg(g) \ge 2;\\
	v(z+\tau)/v(z) \text{  is constant}, & \text{ if}& \deg(g)=1.
	\end{array}
	\right.
	$$
\end{prop}
\begin{proof}
	Denote $q(z):=g(z)-g(z+\tau)$. Then, by periodicity of $F(z)$,
	\begin{equation}\label{period-1}
	v(z+\tau) =e^{q(z)} v(z).
	\end{equation}
	This implies that $ q(z) $ must be polynomial since $ \rho(v)<\infty $. If $\deg(q)= m$, then $g^{(m+1)}(z)$ is periodic, and according to \cite[Lemma 5.1]{YY} we have either $g^{(m+1)}(z)$ is a constant, i.e., $ g(z) $ is polynomial with $ \deg(g)=m+1 $, or $\rho(g)\geq 1$. Now, in the case when $ g(z) $ is polynomial and $ \deg(g) \ge2 $, we obtain from \eqref{period-1} and \cite[Theorem~9.2]{CF} that $$ \rho(v) \ge m+1 = \deg (g). $$ If $ \deg(g) =1 $, then again from \eqref{period-1} we get that $ v(z+\tau)/ v(z) $ is constant.	
\end{proof}

\begin{rem}\label{period-rem-1}
	If $ \rho(v)<1 $ in  Proposition~\ref{period-function-1}, then $ v(z) $ must be constant. Indeed, If $ q(z) $ in \eqref{period-1} is not constant, then we get a contradiction with $ \rho(v)<1 $. Thus, $ q(z) $ must be a constant, and then from \cite[Lemma 3.3]{BL} and \cite[Lemma 5.1]{YY}, $ v(z) $ is constant.
\end{rem}

As a consequence of Proposition~\ref{period-function-1}, we have the following improvement of \cite[Corollary~2]{GY}.

\begin{cor}\label{Cor}
	Let $ f(z) $ be a non-constant periodic entire function with $ \rho_2(f)<1 $.
	Then, for any $ d\in \mathbb{C} $, we have  either $ f(z)=e^{az+b}+d $, for some $ a,b \in \mathbb{C} $, or 
	$$
	\lambda(f-d) =\lambda\left(f^{(k)}\right)= \rho(f) ,\quad \text{for all }\; k\in \mathbb{N}. 
	$$
\end{cor}
	
	\begin{proof}
		Let $ d \in\mathbb{C} $ be arbitrary. Assume that  $ f(z) $ is not of the form $e^{az+b} +d$. Therefore, if $ \lambda(f-d)<\rho(f)$, then we can write
		$$ f(z)-d=v(z)e^{g(z)} , $$
		where $v(z)$ and $g(z)$ are entire functions such that $\rho(v)=\lambda(f-d)<\infty$. Obviously, $g(z)$ is a polynomial, otherwise it follows from Proposition~\ref{period-function-1} that $\rho_2(f)=\rho(g)\geq 1$, which contradicts our assumption. Now, if $\deg(g)\geq 2$, we may use Proposition~\ref{period-function-1} again to obtain
		$$
		\rho(v)<\rho(f)=\deg(g)\leq \rho(v),
		$$
		which is a contradiction, and thus, $\rho(v)<\deg(g)=1$. From this and Remark~\ref{period-rem-1} we conclude that $ v(z) $ must be constant, that is, $ f(z) $ is of the form $ e^{az+b} +d $, which contradicts the assumption.

		Now, let $ k\in \mathbb{N} $ be arbitrary. Clearly, the periodicity of $f^{(k)}(z)$ follows from that of~$f(z)$. Hence, it follows from the previous reasoning, with $ d=0 $, that either 
		$$
		\lambda\left(f^{(k)}\right)=\rho\left(f^{(k)}\right)=\rho(f),
		$$
		or  $ f^{(k)}(z)$ has the form $ f^{(k)}(z) = e^{az+b} $, i.e., $ f(z) $ is of the form
		$$
		f(z)=e^{az+c}+p(z),
		$$
		where $ c $ is a constant (depending on $ a,b $ and $ k $) and $p(z)$ is a polynomial of degree less than~$k$. Since $f(z)$ is periodic, we conclude that $p(z)$ must be constant, which is not possible by the assumption.  Thus, we have the conclusion of Corollary~\ref{Cor}.
	\end{proof}

Next, we present the following improvement of \cite[Theorem 1]{G}.
\begin{prop}\label{G}
	Let $g_1(z),g_2(z),\ldots,g_n(z)$ be non-constant entire functions such that $g_{j}(z)-g_{k}(z)$ are non-constant whenever $j\neq k$, and let $ v_0(z), v_1(z), \ldots, v_n(z) $ be non-zero meromorphic functions with $ \max\limits_{0\le j\le n} \rho(v_j) <1$. If
	\begin{eqnarray}\label{Gr}
	F(z):=\sum_{k=1}^{n} v_{k}(z) e^{g_{k}(z)}+v_{0}(z),
	\end{eqnarray}
	is periodic of period $\tau$, then each $v_{j}(z)$ ($ j=0,\ldots,n $) must be constant, and for every $ k=1, \ldots,n $, $g_k(z)=p_k(z)+a_k z$, where $a_k$ is a constant and $p_k(z)$ is periodic of period~$\tau$. 	
\end{prop}

\begin{proof}
	Notice that $ F(z) $ must be non-constant, otherwise $v_{j}(z)\equiv 0$ ($ j=0,\ldots,n $) by Borel's lemma. We follow the idea of the proof of \cite[Theorem 1]{G}. Since $ F(z)=F(z+\tau) $, it follows from Borel's lemma that, and for any $ k\in \{1,\ldots,n\} $, there exists a fixed integer~$ m_k $ such that
	\begin{equation*}
	v_{k}(z) e^{g_{k}(z)} = v_{k}(z+m_k\tau) e^{g_{k}(z+m_k\tau)}.
	\end{equation*}
	From Proposition~\ref{period-function-1} and Remark~\ref{period-rem-1}, we get that $ v_k(z) $ ($ k=1,\ldots,n $) are constants. Similarly, Borel's lemma implies that $ v_0(z) $ is periodic. Hence, it is constant as $ \rho(v_0)~<~1 $.
\end{proof}

\subsection{Periodicity of polynomials with non-constant coefficients}

The following theorem was the key in proving several new results concerning Conjecture~\ref{Con} and some of its general forms in \cite{LZ}. It is an extension of the classical result of C. and A. R\'enyi \cite[Theorem 2]{RR}.

\begin{theorem}[{\cite[Theorem 2.1]{LZ}}]\label{LZ}
	Let $ f(z) $ be a transcendental entire function such that $N(r,1/f)=S(r,f)$, $ A(z) $ be non-vanishing meromorphic function satisfying $ T(r,A) = S(r,f) $, and let
	$ P(z) $ be a polynomial with at least two non-zero terms. If $ A(z) P\left(f(z)\right)$ is periodic of period $ c $, then $ f(z) $ is $ c $-periodic.
\end{theorem}
	
The aim of this subsection is to extend Theorem~\ref{LZ}, and give a general study on the periodicity of polynomials in $ f(z) $ with non-constant coefficients. This will be so useful in proving our main theorems.

Let $Q(f)$ be a polynomial in $ f(z) $ with at least two non-zero terms and defined as
\begin{equation}\label{poly}
Q(f)=\sum_{s=1}^{l}\alpha_{\nu_s}(z)f(z)^{\nu_s},\quad l \geq 2, \; \nu_{1}<\cdots<\nu_l,
\end{equation}
where $\alpha_{\nu_s}(z)$ are non-vanishing meromorphic functions and small with respect to $f(z)$. 

The next result gives necessary conditions for the periodicity of $ Q(f) $.
\begin{thm}\label{th2}
	Let $f(z)$ be a transcendental entire function with $N(r,1/f)=S(r,f)$ and let $Q(f)$ be as in \eqref{poly} and be a periodic function of period $c$. Then for all $s\in \{1,\ldots,l\}$, the terms $\alpha_{\nu_{s}}(z)f(z)^{\nu_{s}}$ are periodic of period $c$. Furthermore, for any distinct $m,n\in \{1,\ldots,l\}$ for which $\nu_{m}\nu_n>0$, the functions
	\begin{equation}\label{F}
	F_{m,n}(z):=\frac{\alpha_{\nu_{m}}(z) ^{\nu_n}}{\alpha_{\nu_{n}}(z) ^{\nu_m}}
	\end{equation}
	are periodic of period $c$.
\end{thm}

\begin{proof}
	Since $Q(f)$ is periodic of period $c$, it follows that
	\begin{eqnarray}\label{p1}
	\sum_{s=1}^{l}\alpha_{\nu_s}(z+c)f_c^{\nu_s}=\sum_{s=1}^{l}\alpha_{\nu_s}(z)f^{\nu_s},
	\end{eqnarray}
	where $f_c$ stands for $f(z+c)$. Making use of the Valiron-Mohon'ko theorem \cite[Theorem 2.2.5]{L}, it follows $T(r,f_c)\sim T(r,f)$ as $r \to \infty$, probably outside an exceptional set of finite linear measure. Let $m\in \{1,\ldots,l\}$. Then, dividing both sides of \eqref{p1} by $\alpha_{\nu_m}(z)f(z)^{\nu_m}$ yields
	\begin{equation*}
	\sum_{s=1}^l\frac{\alpha_{\nu_s}(z+c)}{\alpha_{\nu_m}(z)}\frac{f_c^{\nu_s}}{f^{\nu_m}}-\sum_{\substack{s=1 \\ s \neq m}}^{l} \frac{\alpha_{\nu_{s}}(z)}{\alpha_{\nu_{m}}(z)} f^{\nu_{s}-\nu_{m}}=1.
	\end{equation*}
	Clearly, for every $s\neq m$,
	$$
	\frac{\alpha_{\nu_s}(z+c)}{\alpha_{\nu_m}(z)}\frac{f_c^{\nu_s}}{f^{\nu_m}}\quad \text{and}\quad \frac{\alpha_{\nu_{s}}(z)}{\alpha_{\nu_{m}}(z)} f^{\nu_{s}-\nu_{m}}
	$$
	cannot be constants. Then, by using \cite[Theorem~1.62]{YY} we obtain
	\begin{equation}\label{p2}
	\frac{\alpha_{\nu_m}(z+c)}{\alpha_{\nu_m}(z)}\left( \frac{f_c}{f}\right) ^{\nu_m}\equiv 1.
	\end{equation}
	Since $m\in \{1,\ldots,l\}$ is arbitrary, we deduce the first assertion that all the terms are periodic of period $ c $.
	The periodicity of the functions \eqref{F} follows directly from \eqref{p2}.
\end{proof}

The conclusions of Theorem~\ref{th2}  cannot be a sufficient condition for the periodicity of $ f(z) $ as shown by the following example.
\begin{example}\label{ex-1}
	For the entire function $f(z)=e^{e^z}\left(\frac{e^z-1}{z}\right)$, the polynomial
	$$
	Q(f)=zf(z) + z^2f(z)^2+z^3f(z)^3
	$$
	is periodic, the coefficients satisfy \eqref{F}, but $ f(z) $ is not periodic.
\end{example}

\begin{rem}
	Gross \cite[Theorem 2]{G} proved that the function $e^{g(z)}+g(z)$ cannot be periodic if $g(z)$ is not periodic. This result is a special case of Theorem~\ref{th2}. Indeed, if $f(z)$ is a transcendental entire function with $N(r,1/f)=S(r,f)$ and $g(z)$ is a small function of $f(z)$, then $f(z)+g(z)$ is periodic of period $c$ if and only if $f(z)$ and $g(z)$ are periodic of period $c$.
\end{rem}

In general, the periodicity of $ f(z) $ is related to the periodicity of the coefficients~of~$Q(f)$. This is shown in the following corollary.

\begin{cor}\label{co-perio}
	Under the hypotheses of Theorem~\ref{th2}, the statements below are equivalent:
	\begin{itemize}
		\item [(i)] One coefficient $\alpha_{\nu_{s}}(z)$ with $\nu_{s}>0$ is $c$-periodic;
		\item [(ii)] All the coefficients are $c$-periodic;
		\item [(iii)] $f(z)$ is $c$-periodic.
	\end{itemize}
\end{cor}
\begin{proof}
	Suppose that one of the coefficients $\alpha_{\nu_s}(z)$ with $\nu_s >0$ is $c$-periodic. Then, by \eqref{p2}, $f(z)^{\nu_s}$ has the same period as $\alpha_{\nu_{s}}(z)$. In fact, if $\alpha_{\nu_{s}}(z)$ is of period $mc$, for some positive integer $m$, then $f(z)$ is either of period $mc$ or $\nu_{s}mc$. This shows that if one of the coefficients $\alpha_{\nu_s}(z)$ with $\nu_s >0$ is $c$-periodic, then $f(z)$ is $c$-periodic either.
	Now, let $f(z)$ be a $c$-periodic function. Then, from \eqref{p2} we see that all the coefficients $\alpha_{\nu_{s}}(z)$ must have the same period as $f(z)$. Therefore, for all $s\in \{1,\ldots,l\}$, $\alpha_{\nu_{s}}(z)$ is $c$-periodic.
\end{proof}

One can see that Theorem~\ref{LZ} can be obtained from Theorem~\ref{th2} and Corollary~\ref{co-perio}.
Indeed, if there exist distinct $ m ,n$ such that $ \nu_m\nu_n>0 $ and $  \alpha_{\nu_{m}}(z) / \alpha_{\nu_{n}}(z) $ is constant, then from \eqref{F}, we see that both $  \alpha_{\nu_{m}}(z) $  and $ \alpha_{\nu_{n}}(z) $ are $c$-periodic, and then $ f(z) $ is $ c $-periodic by Corollary~\ref{co-perio}. 

In particular, if all the coefficients of $Q(f)$ are of order $ <1 $ and $Q(f)-Q(0)$ has at least one constant coefficient, then Corollary~\ref{co-perio} and \cite[Lemma 5.1]{YY} reveal that all the coefficients of $Q(f)$ must be constants.

Next, we show how the growth of the coefficients of $ Q(f) $ is related to the zeros of $ f(z) $ in Theorem~\ref{th2}.

\begin{cor}\label{Cor2}
	In Theorem~\ref{th2}, assume in addition that the coefficients of $Q(f)$ are all entire functions of order $ <1 $, and at least one of them is non-constant. Then $ \lambda(f)\ge 1 $. 
\end{cor}

\begin{proof}
	Assume that $ \lambda(f)<1 $. Then we can write $ f(z) =v(z)e^{g(z)} $, where $ \rho(v)<1 $. By substituting this $ f(z) $ in $ Q(f) $ and using Proposition~\ref{G}, we obtain, for any $ s\in \{1, \ldots,l\} $, that $ \alpha_{\nu_{s}}(z) v(z)^{\nu_{s}} $ is constant. Since both $ \alpha_{\nu_s}(z) $ and $ v(z) $ are entire functions of order $ <1 $, it follows that both of them must be constants, and then all the coefficients are constants, which contradicts the assumption. Thus  $ \lambda(f)\ge 1 $.
\end{proof}

Example~\ref{ex-1} shows that the conclusion of Corollary~\ref{Cor2} is sharp. In the case of meromorphic coefficients, the conclusion of Corollary~\ref{Cor2} does not hold. For example, any non-constant entire function $ v(z) $ with $ \rho(v)<1 $, and take $f(z)=v(z) e^z $. Then the polynomial
$$
Q(f)= \frac{1}{v(z)} f(z) + \frac{1}{v(z)^2}f(z)^2+\frac{1}{v(z)^3}f(z)^3
$$
is periodic. Here, all the coefficients are of order $ <1 $, and $ \lambda(f)<1 $.

%
%

\section{Proofs}\label{Sec 4}

\subsection{Proofs of Theorems~\ref{th3} and \ref{th4}} \label{secp}
Before proceeding with the proofs, we give a preparation on the monomial $ M(z,f) $.

Let $ f(z) $ be a transcendental entire function and let $ d\in\mathbb{C} $ such that $ \lambda(f-d)<\rho(f) $. Then, by Weierstrass factorization theorem, we may write $f(z)=\pi(z) e^{h(z)}+d$, where $h(z)$ is an entire function and $\pi(z)$ is the canonical product of zeros of~$f(z)-d$ with~$\rho(\pi)<\rho(f)$.  Notice, for any positive integer $ k $, that
	$$
	\left( \pi(z) e^{h(z)}\right)^{(k)}	=\left( \pi(z)h'(z)^k+\mathcal{Q}_k(\pi,h')\right) e^{h(z)},
	$$
where $\mathcal{Q}_k(\pi,h')$ is a differential polynomial in $\pi(z)$ and $h'(z)$ with constant coefficients. Here, $\mathcal{Q}_k(\pi,h')$ is non-linear with respect to $ h'(z) $ with degree $\deg_{h'}\mathcal{Q}_k\leq k-1$, and linear with respect to $ \pi(z) $. Therefore, substituting $f(z)=\pi (z)e^{h(z)}+d$ into \eqref{mo} yields
	\begin{eqnarray}\label{c12}
	M(z,f)&=&H(z)\left( 1+\frac{d}{\pi(z)}e^{-h(z)}\right) ^{\lambda_{0}}e^{\gamma_{M}h(z)}.
	\end{eqnarray}
where
	\begin{eqnarray}\label{c}
	H(z)=\pi(z)^{\lambda_{0}}\prod_{k=1}^n\mathcal{L}_k(\pi,h')^{\lambda_{k}}
	\end{eqnarray}
and $\mathcal{L}_k(\pi,h')=\pi(z)\, h'(z)^k+\mathcal{Q}_k(\pi,h')$. Expanding the product in \eqref{c} yields
	\begin{eqnarray}\label{c2}
	H(z)=\pi(z)^{\gamma_M}\,  h'(z)^{\Gamma_M}+Q(\pi,h'),
	\end{eqnarray}
where $Q(\pi,h')$ is a differential polynomial in $\pi(z)$ and $h'(z)$ with constant coefficients, and  $\deg_{h'} Q\leq \Gamma_M-1$.

\begin{proof}[Proof of Theorem~\ref{th3}]
Here we have $ d=0 $, and therefore \eqref{c12} becomes
	$$
	M(z,f)=H(z) e^{\gamma_M h(z)},
	$$
where $ H(z) $ satisfies \eqref{c}	and \eqref{c2}.

\medskip
	
(1) Suppose that $\rho_2(f)<1$. Since $ \rho(\pi)< \rho(f)\le \infty$ and $ \rho(h) <1 $, it follows that $ \rho(H)< \infty $. Hence, from the periodicity of  $ M(z,f) =H(z)e^{\gamma_Mh(z)}$ and Proposition~\ref{period-function-1} we deduce that $ h(z) $ is a polynomial with  $\deg(h)=1$. Indeed, if $ \deg(h)\geq 2$, we may use Proposition~\ref{period-function-1} to obtain
	$$
	\rho(f)=\deg (h)\leq \rho(H)\leq \rho(\pi),
	$$
which is a contradiction. Now, from \eqref{c2} we have $$ \rho(H) \le \rho(\pi)< \rho(f) = \deg(h)=1, $$ and then Remark~\ref{period-rem-1} asserts that $ H (z)$ must be a constant, that is, the product \eqref{c} is a constant. Since all factors of (\ref{c}) are entire functions of order $ <1 $, it follows that all these factors must be constants. Notice that, for each $k$, the differential polynomial $\mathcal{L}_k(\pi,h')$ takes the form
	$$
	\mathcal{L}_k(\pi,h'):=\pi^{(k)}(z)+c_{k-1}\pi^{(k-1)}(z)+\ldots+c_1\pi(z), \quad  \text{where }c_j\text{ are constants},
	$$
and $ \rho(\mathcal{L}_k) \le \rho(\pi)<1 $. If now $\lambda_{0}>0$, then $\pi(z)$ is a constant. Otherwise, there exists a positive $\lambda_{k}$ for which $\mathcal{L}_k(\pi,h')$ is constant. Therefore,
	\begin{equation*}\label{c7}
	\pi^{(k+1)}(z)+c_{k-1}\pi^{(k)}(z)+\ldots+c_1\pi'(z)=0.
	\end{equation*}
If $\pi'(z)\not\equiv 0$, then $ \pi'(z) $ is a non-trivial solution of a linear differential equation with constant coefficients, i.e., $ \pi'(z)$  is a linear combination of functions of the form $ e^{\alpha z} $, where $ \alpha $ is a constant. This yields  $\rho(\pi')=\rho(\pi)=1$, which contradicts $ \rho(\pi)<1 $. Thus~$\pi(z)$ must be a non-zero constant.
Hence, $ f(z) $ takes the form $ e^{az+b} $, where $ a $ and $ b $ are non-zero constants. From the periodicity of $ M(z,f) $, we directly obtain $ e^{\gamma_Mac}=1 $.

\medskip

(2) Suppose now that $1\le \rho_2(f)< \infty$ and $ \lambda(f)<\rho_2(f) $. Then $\rho(\pi)<\rho(h)=\rho_2(f)<~\infty$, and hence $ \rho(H)\le\rho(h)<\infty $. Since $M(z,f)$ is periodic of period $c$, it follows that 
	\begin{eqnarray}\label{c1}
	e^{\gamma_M q(z)}=\frac{H(z)}{H(z+c)}, \quad q(z)=h(z+c)-h(z).
	\end{eqnarray}
As in the proof of Proposition~\ref{period-function-1},  $ q(z) $ is polynomial. If $\deg(q)=t\geq 1$, then applying \cite[Theorem~9.2]{CF} to \eqref{c1} yields
	\begin{eqnarray}\label{c8}
	t+1\leq \rho(H)\leq \rho(h).
	\end{eqnarray}
Making use of \eqref{c2} and  the fact that $ h(z+c) = h(z) + q(z)$, we obtain
	\begin{eqnarray}\label{c5}
	H(z+c)
	=\pi_c(z)^{\gamma_M}(h'(z))^{\Gamma_M}+\tilde{Q}(\pi_c,h'),
	\end{eqnarray}
where $\pi_c(z)$ stands for $\pi(z+c)$, $\tilde{Q}(\pi_c,h')$ is a differential polynomial in $\pi_c(z)$ and $h'(z)$ with polynomial coefficients, and of degree $\deg_{h'}\tilde{Q}\leq \Gamma_M-1$. Combining \eqref{c5} and \eqref{c1} results~in
	\begin{eqnarray}\label{c6}
	(\pi_c(z)^{\gamma_M}-e^{-\gamma_M q(z)} \pi(z)^{\gamma_M})(h'(z))^{\Gamma_M}=e^{-\gamma_Mq(z)}Q(\pi,h')-\tilde{Q}(\pi_c,h').
	\end{eqnarray}
If $\pi_c(z)^{\gamma_M}-e^{-\gamma_M q(z)} \pi(z)^{\gamma_M}\not\equiv 0,$ then by
taking the Nevanlinna characteristic function of both sides of \eqref{c6} and keeping in mind that $\rho=\rho(\pi)<\infty$, we obtain
	\begin{equation*}
	T(r,h')= O(r^{\max\{\rho, t\}})+O(\log r).
	\end{equation*}
Hence, $\rho(h)\leq \max\{\rho, t\}$, which contradicts \eqref{c8} and $ \rho(\pi)<\rho(h) $. Therefore,
	$$
	\pi_c(z)^{\gamma_M}=e^{-\gamma_M q(z)} \pi(z)^{\gamma_M},
	$$
which implies that $f_c(z)^{\gamma_M}=f(z)$. Hence, $f(z)$ is periodic of period $c$ or $\gamma_M c$.

Consider next the case $t=0$. This means that $q(z)=q$ is a constant, hence $h'(z+c)=h'(z)$. From \eqref{c2}, and by making use of \eqref{c1}, we get
	$$
	\left( \pi_c(z)^{\gamma_M}-e^{\gamma_M q}\pi(z)^{\gamma_M}\right) (h'(z))^{\Gamma_M}=e^{\gamma_M q}Q(\pi,h') -Q(\pi_c,h').
	$$
The same arguments as for the case $t\geq 1$ lead to the same conclusion that $f(z)$ is periodic of period $c$ or $\gamma_M c$.
\end{proof} 	


\begin{proof}[Proof of Theorem~\ref{th4}]
Since $\lambda_{0}>0$, it follows from \eqref{c12} that
	\begin{eqnarray}\label{c11}
	M(z,f)=H(z)\sum_{i=0}^{\lambda_{0}} \binom{\lambda_{0}}{i}\left( \frac{d}{\pi(z)}\right) ^{i}e^{(\gamma_M-i)h(z)}.
	\end{eqnarray}
Without loss of generality, we may assume that $\gamma_{M}>\lambda_{0}$, since otherwise $M(z,f)=f(z)^{\gamma_{M}}$, and the periodicity of $f(z)$ follows trivially from that of $M(z,f)$.

Recall that $\rho(\pi)<\rho(f)=\rho(e^h)$. Since $e^{h(z)}$ is of regular order, we deduce by the properties of the limit that $T(r,\pi)=S(r,e^h)$ and consequently $T(r,H)=S(r,e^h)$. Thus, $M(z,f)$ can be regarded as a polynomial (with at least two terms) in $e^{h(z)}$ with coefficients being small with respect to $e^{h(z)}$. Making use of Theorem~\ref{th2} and considering the first two terms of \eqref{c11} result in
	$$
	H(z+c)e^{\gamma_{M}h_c(z)}=H(z)e^{\gamma_{M}h(z)}
	$$
and
	$$
	\frac{H(z+c)}{\pi_c(z)} e^{(\gamma_{M}-1)h_c(z)}=	\frac{H(z)}{\pi(z)}e^{(\gamma_{M}-1)h(z)}.
	$$
From the above two equations, we get that $\pi(z)e^{h(z)}$ is periodic of period $c$. Thus $f(z)$ is periodic of period $c$.
\end{proof}

\subsection{Proofs of Theorems~\ref{th6} and \ref{th7}} 
Similarly to subsection~\ref{secp},  we have $ f(z)=\pi(z) e^{h(z)}+d $ and we see that each monomial $ M_j(z,f) $ in \eqref{dif} satisfies 
	\begin{equation*}
	M_j(z,f)=H_j(z)\left( 1+\frac{d}{\pi(z)}e^{-h(z)}\right) ^{\lambda_{0j}}e^{\gamma_{j}h(z)}.
	\end{equation*}
Here, $ H_j(z) $ has the same form as in \eqref{c} and \eqref{c2}, i.e., 
	\begin{eqnarray}\label{cb}
	H_j(z)=\pi(z)^{\lambda_{0j}}\prod_{k=1}^n\mathcal{L}_k(\pi,h')^{\lambda_{kj}},
	\end{eqnarray}
and 
	\begin{eqnarray}\label{c2b}
	H_j(z)=\pi(z)^{\gamma_j}(h'(z))^{\Gamma_j}+Q_j(\pi,h'),
	\end{eqnarray}
where $\mathcal{L}_k(\pi,h')=\pi(z)\, h'(z)^k+\mathcal{Q}_k(\pi,h')$ and $Q_j(\pi,h')$ is a differential polynomial in $\pi(z)$ and $h'(z)$ with constant coefficients, and  $\deg_{h'} Q_j\leq \Gamma_j-1$. Therefore,
\begin{eqnarray}\label{m}
P(z,f)=\sum_{j=1}^m\alpha_jH_j(z)\left( 1+\frac{d}{\pi(z)}e^{-h(z)}\right) ^{\lambda_{0j}}e^{\gamma_{j}h(z)}.
\end{eqnarray}

\begin{proof}[Proof of Theorem~\ref{th6}]
	
(i)  Assume that $d=0$ and $ \lambda_{0j}>0$ for all $ j $. Then \eqref{m} becomes
	\begin{eqnarray}\label{eqq1}
	P(z,f)&=&\sum_{j=1}^{m}\alpha_{j}H_j(z) e^{\gamma_{j}h(z)}.
	\end{eqnarray}
By using the sequence defined in \eqref{seq}, $P(z,f)$ can be rewritten as:
	\begin{equation}\label{dd7}
	P(z,f)=\sum_{i=1}^{l}\tilde{H}_{i}(z)e^{\delta_{i}h(z)}, \quad 1\leq l\leq m,
	\end{equation}
where
	\begin{equation}\label{dd8}
	\tilde{H}_i(z)=\sum_{\substack{\gamma_j=\delta_i\\ 1\leq j\leq m}}\alpha_jH_j(z),\quad \text{for all}\; i\in \{1,2,\ldots,l\}.
	\end{equation}
Here, we have $ \rho(\tilde{H}_i) \le \max\{\rho(\pi);\rho(h) \}< \rho(e^h)$, and then
	$$
	T(r,\tilde{H}_i)=S(r,e^h), \quad \text{for all}\; i\in \{1,2,\ldots,l\}.
	$$
Thus, $ P(z,f) $ can be seen as a polynomial in $ e^{h(z)} $ with coefficients being small with respect to $ e^{h(z)} $. For the rest of the proof  we use the following lemma.
\begin{lemma}\label{lem}
If $ \lambda_{0j}>0$ for all $ j $, and there exists $ p\in \{1,\ldots,l\} $ such that $ \tilde{H}_p(z) e^{\delta_p h(z)} $ is non-zero periodic, then $ f(z) $ is periodic of period $ c $ or $ \delta_{p}c $.
\end{lemma}
\begin{proof}
 Proposition~\ref{period-function-1} reveals that $ h(z) $ is a polynomial of degree $ \deg(h)=1 $, and $ \rho(\tilde{H}_p)\le \rho(\pi)<\rho(f)=\deg(h)=1 $. Then, Remark~\ref{period-rem-1} reveals that $ \tilde{H}_p(z) $ is a constant. 
From \eqref{dd8} and \eqref{cb}, we have
	\begin{eqnarray}\label{eqq8}
	\sum_{\substack{\gamma_j=\delta_p\\ 1\leq j\leq m}}\alpha_{j}\pi(z)^{\lambda_{0j}}\prod_{k=1}^n\mathcal{L}_k(\pi,h')^{\lambda_{kj}}=\tilde{H}_p.
	\end{eqnarray}
By assumption, we have $\lambda_{0j}>0$ for all $j\in \{1,2,\ldots,m\}$. If $\pi(z)$ is non-constant, then $\pi(z)$ has at least one zero $z_0$. Plugging $z_0$ in \eqref{eqq8} yields $\tilde{H}_p=0$, which is a contradiction. Thus, $ \pi(z) $ doesn't have zeros, and then $\pi(z)$ must be a non-zero constant since $\rho(\pi)<1$. From the periodicity of $ \tilde{H}_p e^{\delta_p h(z)} $, we conclude that $ f(z) $ is periodic of period $ c $ or $ \delta_{p}c $.
\end{proof}
Now, return to the proof of Theorem~\ref{th6}. 
Since $ P(z,f) \not\equiv0$, it follows that there is a $ p\in \{1,\ldots,l\} $ such that $ \tilde{H}_p(z)\not\equiv 0 $.   Now, if $ P(z,f) $ has only one term, i.e.,  
	\begin{equation}\label{1t}
	P(z,f)=\tilde{H}_p(z)e^{\delta_ph(z)},
	\end{equation}
then, by periodicity of $ P(z,f) $ and Lemma~\ref{lem} we deduce that $ f(z) $ is periodic of period $ c $ or $ \delta_{p}c $. If $ P(z,f) $ has at least two non-zero terms, then we may apply Theorem~\ref{th2} to obtain that the term $ \tilde{H}_p(z)e^{\delta_ph(z)} $ is periodic. Again, Lemma~\ref{lem} results in the conclusion that $ f(z) $ is periodic of period $ c $ or $ \delta_{p}c $.

\medskip

(ii)  Assume now that $d\neq 0$ and all $\lambda_{0j}$ are equal to the same positive integer $\lambda$. Then, we may write \eqref{m} as
	\begin{eqnarray}\label{eqq12}
	P(z,f)&=&\left( 1+\frac{d}{\pi(z)}e^{-h(z)}\right) ^{\lambda}\sum_{i=1}^{l}\tilde{H}_i(z)e^{\delta_{i}h(z)},
	\end{eqnarray}
where $\tilde{H}_i(z)$ are given in \eqref{dd8} and the $\delta_i$'s are defined in \eqref{seq}. Clearly, the coefficients $\tilde{H}_i(z)$ are not all vanishing identically, since otherwise $P(z,f)\equiv 0$, which is a contradiction. If the sum in \eqref{eqq12} reduces to a constant, then  Theorem~\ref{th2} reveals directly that $ f(z) $ is $ c $-periodic.  Otherwise, expanding \eqref{eqq12} yields
	\begin{eqnarray}\label{n}
\!\!\!\!\!\!	P(z,f)=\sum_{i=1}^{l}\tilde{H}_i(z)e^{\delta_{i}h(z)}+d\lambda\sum_{i=1}^{l}\frac{\tilde{H}_i(z)}{\pi(z)}e^{(\delta_{i}-1)h(z)}+\ldots+d^{\lambda}\sum_{i=1}^{l}\frac{\tilde{H}_i(z)}{\pi(z)^{\lambda}}e^{(\delta_{i}-\lambda)h(z)}.
	\end{eqnarray}
Let $ p\in\{1,\ldots,m\} $ be the largest index for which $\tilde{H}_p(z)\not\equiv 0$. Then the leading term of $P(z,f)$  would be $\tilde{H}_p(z)e^{\delta_ph(z)}$. It follows by applying Theorem~\ref{th2} that $\tilde{H}_p(z)e^{\delta_ph(z)}$ is periodic of period $c$, and then Lemma~\ref{lem} gives that $f(z)$ is periodic of period $c$~or~$\delta_{p}c$.

\medskip

(iii) Assume that $ d\neq 0 $, $ \lambda_{0j}>0$ for all $ j $ and $ \gamma_p \in \Lambda_{P} $.  Recall that $\delta_l=\gamma_P$. It follows from the following lemma that $ \tilde{H}_l(z) $ defined in \eqref{dd8} satisfies $\tilde{H}_l(z)\not\equiv 0$.

\begin{lemma}\label{lem2}
If there exists  $ s $ such that $\delta_s\in \Lambda_P$, then $\tilde{H}_l(z)\not\equiv 0$.
 \end{lemma}
\begin{proof}
By making use of \eqref{c2b}, we obtain
\begin{eqnarray}\label{dd9}
\tilde{H}_s(z)&=&\sum_{\substack{\gamma_j=\delta_s\\ 1\leq j\leq m}}\alpha_jH_j(z)\nonumber\\
&=&\sum_{\substack{\gamma_j=\delta_s\\ 1\leq j\leq m}} \Big( \alpha_j   \pi(z)^{\delta_s}(h'(z))^{\Gamma_j}+Q_j(\pi,h') \Big) \nonumber\\
&=& \beta_s\pi(z)^{\delta_s}(h'(z))^{\tilde{\Gamma}_s}+\tilde{Q}_s(\pi,h'),
\end{eqnarray}
where $$\beta_s=\sum\limits_{j\in\Lambda(\delta_s)}\alpha_j\neq 0,$$ and $\tilde{\Gamma}_s$ is the highest weight of the terms with degree $\delta_{s}$ in \eqref{dif}, $\tilde{Q}_s(\pi,h')$ is a differential polynomial in $h'(z)$ and $ \pi(z) $ with $\deg_{h'}\leq \tilde{\Gamma}_s-1$, with constant coefficients. 
If $\tilde{H}_s(z)\equiv 0$, then from \eqref{dd9} we get that $ T(r,h') = O(T(r,\pi)) $, and this yields $ \rho(h)\le \rho(\pi) $, which contradicts the assumption $ \rho(\pi)<\rho(h) $. Thus $\tilde{H}_s(z)\not\equiv 0$. This completes the proof the lemma.
\end{proof}

From \eqref{m} we obtain that
	\begin{eqnarray}
	P(z,f)&=&\sum_{j=1}^m\alpha_{j}H_j(z)\left(e^{\gamma_{j}h(z)}+\ldots+\left( \frac{d}{\pi(z)}\right) ^{\lambda_{0j}}e^{(\gamma_{j}-\lambda_{0j})h(z)} \right) \nonumber\\
	&=&\tilde{H}_l(z)e^{\delta_lh(z)}+\mathcal{D}(z,e^{h(z)}), \label{eee}
	\end{eqnarray}
where $\mathcal{D}(z,e^{h(z)})$ is a polynomial in $e^{h(z)}$ of degree $\leq \delta_l-1$, with coefficients being small with respect to $e^{h(z)}$. Therefore, by using Theorem~\ref{th2}, we obtain that $ \tilde{H}_l(z)e^{\delta_lh(z)} $ is periodic of period $ c $, and then Lemma~\ref{lem}  asserts that $f(z)$ is periodic of period $c$~or~$\delta_{p}c$.
\end{proof}

%
%
%

\begin{proof}[Proof of Theorem~\ref{th7}]
(i) Assume that $d=0$. In this case, $ P(z,f) $ takes the form~\eqref{dd7}. Since $\Lambda_P\neq \emptyset$, there exists $\delta_s\in \Lambda_P$, and hence $ \tilde{H}_s(z)\not\equiv 0 $.  Keeping in mind that $T(r,\tilde{H}_i)=S(r,e^h)$, for all $i=1,\ldots,l$, we may apply Theorem~\ref{th2} to \eqref{dd7} to obtain that $ \tilde{H}_s(z)e^{\delta_s h(z)} $ is periodic of period $ c $. Following the same reasoning as in the proof of Theorem~\ref{th3}(2),  we conclude that $f(z)$ is either periodic of period $c$~or~$\delta_s c$.
		
\medskip

(ii)  Assume now that $d\neq 0$.  If $\lambda_{0j}=0$ for all $j=1,\ldots,m$, then $P(z,f)$ takes the form \eqref{dd7}, and as in the previous case (i), we conclude the periodicity of $f(z)$. Otherwise, if  $ \gamma_{P}\in \Lambda_{P} $, then by recalling that $\delta_l=\gamma_P \in \Lambda_P$, we get $\tilde{H}_l(z)\not\equiv 0$ by Lemma~\ref{lem2}. Applying Theorem~\ref{th2} into \eqref{eee}, we obtain that $ \tilde{H}_l(z)e^{\delta_lh(z)} $ is periodic of period $ c $. Again, as in Case~(i) we obtain that $ f(z) $ is periodic of period $ c $ or $ \gamma_P c $.
\end{proof}


\end{document}